\documentclass{amsproc}

\usepackage[latin1]{inputenc}
\usepackage{aeguill}
\usepackage{amsfonts}
\usepackage{amssymb}
\usepackage{xspace}
\usepackage{amsmath}
\usepackage{amsthm}
\usepackage{dsfont}
\usepackage{mathabx}
\usepackage[all,cmtip]{xy}
\usepackage{enumitem}
\usepackage[mathlines]{lineno}
\usepackage{amsmath}

\theoremstyle{definition}

\theoremstyle{remark}

\numberwithin{equation}{section}

\newtheorem{thm}{Theorem}[section]
\newtheorem{prop}{Proposition}[section]
\newtheorem{df}{Definition}[section]
\newtheorem{lem}{Lemma}[section]

\newtheorem{hyp}{Hypothesis}[section]
\newtheorem{ex}{Example}[section]
\newtheorem{conj}{Conjecture}[section]
\newtheorem{rem}{Remark}[section]

\newenvironment{dem}{\paragraph{Proof}}
{\begin{flushright}$\Box$\end{flushright}}

\newcommand{\N}{\mathbb{N}}
\newcommand{\Z}{\mathbb{Z}}
\newcommand{\Q}{\mathbb{Q}}
\newcommand{\C}{\mathbb{C}}

\newcommand{\R}{\mathbb{R}}

\newcommand{\AK}{\mathbb{A}_{K}}
\newcommand{\AF}{\mathbb{A}_{F}}

\newcommand{\AL}{\mathbb{A}_{L}}

\newcommand{\AQ}{\mathbb{A}_{\mathbb{Q}}}

\begin{document}

\title{An automorphic variant of the Deligne conjecture}

\author{Jie LIN}
\address{Institut de Mathématiques de Jussieu}
\curraddr{Case 247, doctorant, Institut de Mathématiques de Jussieu, 4, place Jussieu, 75005 Paris}
\email{jie.lin@imj-prg.fr}
\thanks{The author wants to thank her advisor, Michael HARRIS, for his advice and guidance.}

\date\today

\dedicatory{}


\begin{abstract}
In this paper we introduce an automorphic variant of the Deligne conjecture for tensor product of two motives over a quadratic imaginary field. On one hand, we define some motivic periods and rewrite the Deligne conjecture in terms of these periods. On the other hand, we define the automorphic analogue of these motivic periods and then give a purely automorphic variant of the Deligne conjecture. At last, we introduce some known results of this automorphic variant.
\end{abstract}

\maketitle

\section*{Introduction}
The goal of this paper is to formulate an automorphic variant of the Deligne conjecture and then introduce some known results

In \cite{deligne79}, P. Deligne has constructed complex invariants for motives over $\Q$ and conjectured that special values of motivic $L$-functions are related to thess invariants. Inspired by this conjecture, several results have been revealed in the automorphic setting recently.

In this paper, we are interested in the Rankin-Selberg $L$-function for automorphic pairs, i.e. the tensor product $L$-function for $GL_{n}\times GL_{n'}$ where $n$ and $n'$ are two positive integers. We may assume that $n\geq n'$. 

The first case was treated in \cite{harris97} where $n'=1$. In his article, M. Harris defined some automorphic periods for certain cuspidal representations over quadratic imaginary fields. He proved that the special values of such automorphic representation twisted by a Hecke character can be written in terms of these periods.

The next attempt is done in the year $2013$ where $n'=n-1$ with a local condition on the infinity type. The main formula is given in \cite{harrismotivic} and then simplified in \cite{lincomptesrendus}.

After more cases are studied in the author's thesis, a concise formula on the relations between special values of automorphic $L$-functions and automorphic periods is found out. We state it in Theorem \ref{main theorem} of this paper. It is natural to raise the following question: is this formula compatible with the Deligne conjecture?

The answer is yes. There are already several discussions on motivic periods and the Deligne conjecture in Harris's papers. In \cite{harrisadjoint}, he studied the Deligne conjecture for tensor product of two conjugate self-dual motives over quadratic imaginary fields. He has constructed some motivic periods and reformulated the Deligne conjecture in terms of these motivic periods.

In the author's thesis, the conjugate self-dual condition is dropped. Moreover, the formula is simplified by defining some new motivic periods. We see directly that if we identify the new motivic periods with the automorphic periods, the automorphic results we have obtained (c.f. equation \ref{equation automorphic}) are exactly the same as predicted by the Deligne conjecture (c.f. equation \ref{equation motivic}).

Hence, to show the compatibility, it remains to compare the two types of periods. We discuss this in the first half of Section $2.4$ of this paper. More details can be found in Section $9.2$ of \cite{linthesis}. Some ideas are also explained in \cite{harrismotivic}.

We remark that nothing is proved in this paper for the Deligne conjecture itself. Relation $2.2$, predicted by the Tate conjecture, is still conjectural and is very difficult to prove. In an ongoing work of the author in collaboration with Harald Grobner and Michael Harris, we are trying to prove this relation by assuming the Ichino-Ikeda conjecture.

\bigskip

This paper contains two parts. We state the motivic results in the first part. We first explain the Deligne conjecture in Section $1.1$. We then construct the motivic periods for motives over a quadratic imaginary field in Section $1.2$. We reformulate the Deligne conjecture for tensor product of two motives over quadratic imaginary fields in Section $1.3$ and simplify the formula in Section $1.4$.

The second part is devoted to automorphic results. We discuss the conjectural relations between motives and automorphic representations in Section $2.1$. We then introduce the theory of base change in Section $2.2$ which is inevitable in our method. We construct the automorphic periods in Section $2.3$ and formulate the automorphic variant of the Deligne conjecture in Section $2.4$. We introduce some known results for this automorphic variant at last. 

The motivic part is complete and self-contained. But due to limitation of space, some details and most proof are not provided in the automorphic part. We sincerely apologize for that and refer the reader to the references and forthcoming papers of the author. 

\section*{Basic Notation}

We fix an algebraic closure $\overline{\Q}\hookrightarrow \C$ of $\Q$.

We fix $K\hookrightarrow \overline{\Q}$ an embedding of a quadratic imaginary field into $\overline{\Q}$. 

We denote by $c$ the complex conjugation on $\C$. Via the fixed embedding $\overline{\Q}\hookrightarrow \C$, it can be considered as an element in $Gal(\overline{\Q}/\Q)$.

For any number field $L$, let $\AL$ be the adele ring of $L$. We denote by $\Sigma_{L}$ the set of embeddings from $L$ to $\overline{\Q}$. 

Let $E$ be a number field. Let $A$, $B$ be two elements in $E\otimes_{\Q}\C$. We say $A\sim_{E} B$ if either A=0, or $B=0$, or $B\in (E\otimes_{\Q}\C)^{\times}$ and $A/B\in E^{\times} \subset (E\otimes_{\Q}\C)^{\times}.$ We say $A\sim_{E;K} B$ if either A=0, or $B=0$, or $B\in (E\otimes_{\Q}\C)^{\times}$ and $A/B\in (E\otimes_{\Q} K)^{\times} \subset (E\otimes_{\Q}\C)^{\times}.$ 

Finally, we identify $E\otimes_{\Q}\C$ with $\C^{\Sigma_{E}}$ where $e\otimes z$ with $e\in E$, $z\in \C$, is identified with $(\sigma(e)z)_{\sigma\in\Sigma_{E}}$.

\section{Motives and the Deligne conjecture}\label{chapter motive}
\bigskip
\subsection{Motives over $\Q$}
We first summarize the basic notation and construction for the Deligne conjecture in \cite{deligne79}. 

In this article, a \textbf{motive} simply means a pure motive for absolute Hodge cycles. More precisely, a motive $M^{\#}$ over $\Q$ with coefficients in a number field $E$ is given by its Betti realization $M^{\#}_{B}$, its de Rham realization $M^{\#}_{DR}$ and its $l$-adic realizations $M^{\#}_{l}$ for all prime numbers $l$ where $M^{\#}_{B}$ and $M^{\#}_{DR}$ are finite dimensional vector space over $E$, $M^{\#}_{l}$ is a finite dimensional vector space over $E_{l}:=E\otimes_{\Q} \Q_{l}$ endowed with:
\begin{itemize}
\item $I_{\infty}: M^{\#}_{B}\otimes \C \xrightarrow{\sim} M^{\#}_{DR}\otimes \C $ as free $E\otimes_{\Q}\C$-module;
\item $I_{l}: M^{\#}_{B}\otimes \Q_{l}\xrightarrow{\sim} M^{\#}_{l}$ as free $E\otimes _{\Q} \Q_{l}$-module.

\end{itemize}

From the isomorphisms above, we see that $dim_{E}M^{\#}_{B}=dim_{E}M^{\#}_{DR}=dim_{E_{l}}M^{\#}_{l}$ and this is called the \textbf{rank} of $M^{\#}$, denoted by $\text{rank}(M)$. We have moreover:
\begin{enumerate}
\item An $E$-linear involution (infinite Frobenius) $F_{\infty}$ on $M^{\#}_{B}$ and a Hodge decomposition $M^{\#}_{B}\otimes \C=\bigoplus\limits_{p,q\in \Z}(M^{\#})^{p,q}$ as free $E\otimes \C$-module such that $F_{\infty}$ sends $(M^{\#})^{p,q}$ to $(M^{\#})^{q,p}$.

For $w$ an integer, we say $M^{\#}$ is \textbf{pure of weight} $w$ if $(M^{\#})^{p,q}=0$ for $p+q\neq w$. 

We define the \textbf{Hodge type} of $M^{\#}$ by the set $T(M^{\#})$ consisting of pairs $(p,q)$ such that $(M^{\#})^{p,q}\neq 0$. We remark that if $(p,q)$ is an element of $T(M^{\#})$, then $(q,p)$ is also contained in $T(M^{\#})$.

We define the \textbf{Hodge numbers} by $h_{p,q}:=dim_{E\otimes \C} (M^{\#})^{p,q}$ for $p,q\in\Z$. We say $M^{\#}$ is \textbf{regular} if $h_{p,q}\leq 1$ for all $p,q\in \Z$.\\

\item An $E$-rational Hodge filtration on $M^{\#}_{DR}$: $\cdots \supset (M^{\#})^{i}\supset (M^{\#})^{i+1}\supset \cdots$ which is compatible with the Hodge structure on $M^{\#}_{B}$ via $I_{\infty}$, i.e.,
\begin{equation}
\nonumber I_{\infty}(\bigoplus\limits_{p\geq i}(M^{\#})^{p,q})=(M^{\#})^{i}\otimes \C.
\end{equation}
\item A Galois action of $G_{\Q}$ on each $M^{\#}_{l}$ such that the family $\{ M_{l}^{\#}\}_{l}$ forms a compatible system of $l$-adic representations $\rho_{l}:G_{\Q} \longrightarrow GL(M^{\#}_{l})$. More precisely, for each prime number $p$, let $I_{p}$ be the inertia subgroup of a decomposition group at $p$ and $F_{p}$ be the geometric Frobenius of this decomposition group. For all $l\neq p$, the polynomial $det(1-F_{p}\mid (M^{\#}_{l})^{I_{p}})$ has coefficients in $E$ and is independent of the choice of $l$.
We can then define the local $L$-factor $L_{p}(s,M^{\#}):=det(1-p^{-s}F_{p}|(M^{\#}_{l})^{I_{p}})^{-1}\in E(p^{-s})$ by taking whatever $l\neq p$.
\end{enumerate}

\bigskip

For any fixed embedding $\sigma: E\hookrightarrow \C$, we may consider $L_{p}(s,M^{\#},\sigma)$ as a complex valued function. We define $L(s,M^{\#},\sigma)=\prod\limits_{p}L_{p}(s,M^{\#},\sigma)$. It converges for $Re(s)$ sufficiently large. It is conjectured that this $L$-function has analytic continuation and functional equation on the whole complex plane.

We can also define $L_{\infty}(s,M^{\#})$, the infinite part of the $L$-function, as in Section $5.3$ of \cite{deligne79}. We shall give the precise definition under certain hypothesis later. 

The Deligne conjecture studies the critical values of motivic $L$-functions. We state first the definition for critical points and then give a simple criteria under Hypothesis \ref{hyp1}.

\begin{df}
We say an integer $m$ is \textbf{critical} for $M^{\#}$ if neither $L_{\infty}(s,M^{\#})$ nor $L_{\infty}(1-s,\check{M})$ has a pole at $s=m$ where $\check{M^{\#}}$ is the dual of $M^{\#}$. We call $m$ a \textbf{critical point} of $M^{\#}$ and $L(m,M^{\#})$ a \textbf{critical value} of the $L$-function for $M^{\#}$.
\end{df}

A necessary condition for the existence of critical points is that the infinite Frobenius $F_{\infty}$ acts as a scalar at $(M^{\#})^{p,p}$ for every integer $p$. This condition is automatically satisfied when $M^{\#}$ has no $(p,p)$ class. In fact, this is the only case that we will meet throughout the text. There is no harm to assume from now on that:
\begin{hyp}\label{hyp1}
The motive $M^{\#}$ has no $(p,p)$-class, i.e., $(M^{\#})^{p,p}=\{0\}$ for any integer $p$. 
\end{hyp}

We denote the normalized Gamma function by $\Gamma_{\C}(s):=2(2\pi )^{-s}\Gamma(s)$ where $\Gamma$ means the normal Gamma function. If $M^{\#}$ satisfies Hypothesis \ref{hyp1}, then $L_{\infty}(s,M^{\#})$ is defined as
\begin{equation}
L_{\infty}(s,M^{\#}):=\prod\limits_{(p,q)\in T(M^{\#}),p<q}\Gamma_{\C}(s-p)^{h_{p,q}}.
\end{equation}

Since the poles of the Gamma function are the non positive integers, it is easy to deduce that:

\begin{lem} \label{critical existence}
For an motive $M^{\#}$ satisfying Hypothesis \ref{hyp1}, an integer $m$ is critical if and only if for any $(p,q)\in T(M^{\#})$ such that $p<q$, we have $p<m<q+1$. In particular, the motive $M^{\#}$ always has critical points.
\end{lem}

\bigskip

The Deligne conjecture predicts a relation between critical values and the Deligne period. We now define the Deligne period.

\begin{df}
Let $M^{\#}$ be a motive satisfying Hypothesis \ref{hyp1}. We denote by $(M^{\#}_{B})^{+}$ (resp. $(M^{\#}_{B})^{-}$) the subspace of $M^{\#}_{B}$ fixed by $F_{\infty}$ (resp. $-F_{\infty}$). We set $F^{+}(M^{\#})=F^{-}(M^{\#}):=F^{\omega/2}(M)$, a subspace of $M^{\#}_{DR}$. It is easy to see that $I_{\infty}^{-1}(F^{+}(M^{\#})\otimes \C)$ equals $\bigoplus\limits_{p>q}(M^{\#})^{p,q}$.

The comparison isomorphism then induces an isomorphism:
\begin{equation}
(M^{\#}_{B})^{+}\otimes \C \hookrightarrow M^{\#}_{B}\otimes\C \xrightarrow{\sim} M^{\#}_{DR}\otimes \C\rightarrow (M^{\#}_{DR}/F^{+}(M^{\#}))\otimes \C. \end{equation}
The \textbf{Deligne period} $c^{+}(M)$ is defined to be the determinant of the above isomorphism with respect to any fixed $E$-bases of $(M^{\#}_{B})^{+}$ and $M^{\#}_{DR}/F^{+}(M^{\#})$. It is an element in $(E\otimes_{\Q}\C)^{\times}$ and is well defined up to multiplication by elements in $E^{\times}$. Similarly, we may define $c^{-}(M^{\#})$. 
\end{df}

At last, we consider $L(m,M^{\#})=L(m,M^{\#},\sigma)_{\sigma\in\Sigma_{E}}\in \C^{\Sigma_{E}}$ as an element in $E\otimes \C$.\\

\begin{conj}\label{deligne conjecture}\textbf{The Deligne conjecture}

We define $d^{\pm}:=dim_{E}((M^{\#}_{B})^{\pm})$. If $m$ is critical for $M^{\#}$ then 
\begin{equation}
L(m,M^{\#})\sim_{E} (2\pi i)^{d^{\epsilon}m}c^{\epsilon}(M^{\#})
\end{equation}
where $\epsilon$ is the sign of $(-1)^{m}$.
\end{conj}

\bigskip
We remark that $d^{+}=d^{-1}$ is simply $\text{rank}(M)/2$ under Hypothesis \ref{hyp1}.

\bigskip 
\subsection{Motive over quadratic imaginary field}

Recall that $K$ is a quadratic imaginary field with fixed embedding $K\hookrightarrow \overline{\Q}$. 

Let $E$ be a number field. In this section, we shall consider a motive $M$ over $K$ with coefficients in $E$. For each embedding $K\hookrightarrow \C$, the motive $M$ has realizations as in the previous section. We keep the notation $M$ to indicate the realizations with respect to the fixed embedding $K\hookrightarrow \C$. We use $M^{c}$ to indicate the realizations respect to the complex conjugation of the fixed embedding. The infinite Frobenius gives an E-linear isomorphism $F_{\infty}: M_{B} \xrightarrow{\sim} M_{B}^{c}.$

We assume that $M$ is regular and pure of weight $\omega(M)$. We write $n$ for the rank of $M$. Since $M$ is regular, we can write its Hodge type as $\{(p_{i},q_{i})\mid 1\leq i\leq n\}$ with $p_{1}>p_{2}>\cdots>p_{n}$. We have $q_{i}=\omega(M)-p_{i}$ for all $i$. The Betti realization $M_{B}$ has a Hodge decomposition $M_{B}\otimes_{\Q}\C=\bigoplus\limits_{i=1}^{n}M^{p_{i},q_{i}}$ as $E\otimes_{\Q}\C$-module.

The Hodge type of $M^{c}$ then equals $\{(p_{i}^{c},q_{i}^{c})\mid 1\leq i\leq n\}$ where $q_{i}^{c}=\omega(M)-p_{i}^{c}$ and $p_{i}^{c}=q_{n+1-i}=\omega(M)-p_{n+1-i}$. We write in this way so that the Hodge numbers $p_{i}^{c}$ are still in decreasing order. We know $M_{B}^{c}\otimes_{\Q}\C=\bigoplus\limits_{i=1}^{n}(M^{c})^{p_{i}^{c},q_{i}^{c}}$ and $F_{\infty}$ induces $E$-linear isomorphisms: $M^{p_{i},q_{i}}\xrightarrow{\sim} (M^{c})^{p^{c}_{n+1-i},q^{c}_{n+1-i}}$.

The De Rham realization $M_{DR}$ is still a finite dimensional $E$-linear space endowed with an $E$-rational Hodge filtration $M_{DR}=M^{p_{n}}\supset M^{p_{n-1}}\supset\cdots\supset M^{p_{1}}$. The comparison isomorphism:
\begin{equation}
I_{\infty}: M_{B}\otimes_{\Q} \C\xrightarrow{\sim} M_{DR}\otimes_{\Q}\C
\end{equation}
induces compatibility isomorphisms on the Hodge decomposition of $M_{B}$ and the Hodge filtration on $M_{DR}$.

\begin{df}
For any fixed $E$-bases of $M_{B}$ and $M_{DR}$, we can extend them to $E\otimes_{\Q}\C$ bases of $M_{B}\otimes \C$ and $M_{DR}\otimes \C$ respectively. Such bases are called \textbf{$E$-rational}, or simply rational if the number field $E$ is clear. We define $\delta(M)$ to be the determinant of $I_{\infty}$ with respect to any $E$-rational bases and call it the \textbf{determinant period}. It is an element in $(E\otimes\C)^{\times}$well defined up to multiplication by elements in $E^{\times}\subset (E\otimes\C)^{\times}$.
\end{df} 

This period is defined in (1.2.2) of \cite{harris97} and in (1.2.4) of \cite{harrisadjoint}. It is an analogue of Deligne's period $\delta$ defined in ($1.7.3$) of \cite{deligne79} for motives over $\Q$.

\bigskip

Let us now fix some bases. We take $\{e_{i}\mid{1\leq i\leq n}\}$ an $E$-base of $M_{B}$. Since $F_{\infty}$ is $E$-linear on $M_{B}$, we know $\{e^{c}_{i}:=F_{\infty}e_{i}\mid 1\leq i\leq n\}$ forms an $E$-base of $M_{B}^{c}$.

Recall that $I_{\infty}$ is compatible with the Hodge structures. We have, for each $1\leq i\leq n$, an isomorphism:
\begin{equation}\label{comparison Hodge De Rham}
I_{\infty}: \bigoplus\limits_{p_{j}\geq p_{i}}M^{p_{j},q_{j}}=\bigoplus\limits_{j\leq i}M^{p_{j},q_{j}}\xrightarrow{\sim} M^{p_{i}}\otimes_{\Q}\C.
\end{equation}

Therefore, for each $i$, the comparison isomorphism $I_{\infty}$ induces an isomorphism 
\begin{equation}
M^{p_{i},q_{i}}\xrightarrow{\sim} M^{p_{i}}\otimes_{\Q}\C/ M^{p_{i-1}}\otimes_{\Q}\C.
\end{equation}
 Here we set $M^{p_{0}}=\{0\}$. 
 
 Let $\omega_{i}$ be a non zero element in $M^{p_{i},q_{i}}$ such that the image of $\omega_{i}$ by the above isomorphism is in $M^{p_{i}} (\text{mod }M^{p_{i-1}}\otimes_{\Q}\C)$. In other words, $I_{\infty}(\omega_{i})$ is equivalent to an element in $M^{p_{i}}$ modulo $M^{p_{i-1}}\otimes_{\Q}\C$. 

Note that $M^{p_{i},q_{i}}$ is a free $E\otimes \C$-module of rank 1. It follows that the family $\{\omega_{i}\}_{1\leq i\leq n}$ forms an $E\otimes\C$-base of $M_{B}\otimes \C$. Therefore the family $\{I_{\infty}(\omega_{i})\}_{1\leq i\leq n}$ forms an $E\otimes\C$-base of $M_{DR}\otimes \C$. This base is not rational, i.e. is not contained in $M_{DR}$. But by the above construction, it can pass to a rational base of $M_{DR}\otimes_{\Q}\C$ with a unipotent matrix by change of basis. Since the determinant of a unipotent matrix is always one, we can use this base to calculate $\delta(M)$. 

We define $\omega^{c}_{i}\in (M^{c})^{p_{i}^{c},q_{i}^{c}}$ similarly. We will use $\{I_{\infty}(\omega_{i}^{c})\mid 1\leq i\leq n\}$ as an $E\otimes_{\Q} \C$ base of $M_{DR}\otimes \C$ to calculate motivic periods henceforth.

\begin{df}
We know $(M^{c})^{p_{n+1-i}^{c},q_{n+1-i}^{c}}$ is a rank one free $E\otimes\C$-module and $F_{\infty}\omega_{i}$ is a non-zero element in it. Hence for each $1\leq i\leq n$ there exist a number $Q_{i}(M)\in (E\otimes\C)^{\times}$ such that
$F_{\infty}\omega_{i}=Q_{i}(M)\omega^{c}_{n+1-i}$. These numbers in $(E\otimes \C)^{\times}$ are called \textbf{motivic periods} and are well defined up to multiplication by elements in $E^{\times}$.
\end{df}

Since $F_{\infty}^{2}=Id$, we have $F_{\infty}\omega^{c}_{n+1-i}=Q_{i}(M)^{-1}\omega_{i}$. We deduce that:

\begin{lem}\label{period for Mc}
For all $1\leq i\leq n$, $Q_{i}(M^{c})\sim_{E} Q_{n+1-i}(M)^{-1}$.
\end{lem}

\bigskip

Now that the bases are fixed, we can write down the coefficients of certain vectors and then calculate the Deligne period. We write $\omega_{a}=\sum\limits_{i=1}^{n}A_{ia}e_{i}$, $\omega_{t}^{c}=\sum\limits_{i=1}^{n}A_{it}^{c}e^{c}_{i}$ with $A_{ia}, A_{it}^{c}\in E\otimes \C$ for all $1\leq i,a,t\leq n$.

We know $\delta(M)^{-1}=det(A_{ia})_{1\leq i,a \leq n}$. This implies that $\bigwedge_{i=1}^{n}\omega_{i}=\delta(M)^{-1}\bigwedge_{i=1}^{n}e_{i}$.

We denote by $det(M)$ the determinant motive of $M$ as in section $1.2$ of \cite{harrisadjoint}. We know $I_{\infty}(\bigwedge_{i=1}^{n}\omega_{i})$ is an $E$-base of $det(M)_{DR}$ and $\bigwedge_{i=1}^{n}e_{i}$ is an $E$-base of $det(M)_{B}$. It is easy to deduce by the definition of motivic periods that \begin{equation}
F_{\infty}(\bigwedge_{i=1}^{n}\omega_{i})=\prod\limits_{1\leq i\leq n}Q_{i}(M)\bigwedge_{i=1}^{n}\omega_{i}^{c}.
\end{equation} Therefore, we can get the following lemma:

\begin{lem}\label{determinant motive lemma}
There are relations between motivic periods:
\begin{eqnarray}
\delta(M) \sim_{E} \delta(det(M))\\
Q_{1}(det(M))\sim_{E} \prod\limits_{i=1}^{n}Q_{i}(M)
\end{eqnarray}
\end{lem}

\bigskip

\begin{lem}\label{delta c}
For all motive $M$ as above, we have:
\begin{equation}\nonumber
\delta(M^{c})\sim_{E}(\prod\limits_{1\leq i\leq n}Q_{i})\delta(M).
\end{equation}
\end{lem}
\begin{dem}
This follows directly from equation (\ref{coefficient relation}). 

One can also prove this with help of Lemma \ref{determinant motive lemma}. In fact, by Lemma \ref{determinant motive lemma}, we may assume that $n=1$. We take $\omega\in M_{DR}$, $\omega^{c} \in M^{c}_{DR}$ and $e\in M_{B}$ as before. Then $e=\delta(M)\omega$ and $e^{c}=\delta(M^{c})\omega^{c}$ where $e^{c}=F_{\infty}e$.

By definition of motivic period, we have $F_{\infty}\omega=Q_{1}(M)\omega^{c}$ and then $\omega^{c}=Q_{1}(M)^{-1}F_{\infty}\omega=Q_{1}(M)^{-1}F_{\infty}(\delta(M)^{-1}e)=Q_{1}(M)^{-1}\delta(M)^{-1}e^{c}$. It follows that $\delta(M^{c})\sim_{E}Q_{1}(M)\delta(M)$ as expected.
\end{dem}

\bigskip

\begin{ex}\textbf{Tate motive}
\text{}

Let $\Z(1)_{K}$ be the extension of $\Z(1)$ from $\Q$ to $K$. It is a motive with coefficients in $K$. As in section $3.1$ of \cite{deligne79}, $\Z(1)_{K,B}=H_{1}(\mathbb{G}_{m,K})\cong K$ and $\Z(1)_{K,DR}$ is the dual of $H^{1}_{DR}(\mathbb{G}_{m,K})$ with generator $\cfrac{dz}{z}$. Therefore the comparison isomorphism $\Z(1)_{K,B}\otimes \C\cong K\otimes \C\rightarrow \Z(1)_{K,DR}\otimes \C\cong K\otimes \C$ sends $K$ to $\oint\cfrac{dz}{z}K=(2\pi i )K$. We have $\delta(\Z(1)_{K})\sim_{K} 2\pi i$. 

In general, let $M$ be a motive over $K$ with coefficients in $E$ of rank $r$. We have 
\begin{equation}\label{Tate lemma}
\delta(M(n))\sim_{E;K} (2\pi i)^{nr}\delta(M).
\end{equation}
\end{ex}

\bigskip \subsection{The Deligne period for tensor product of motives}

We now consider the Deligne period for tensor product of motives over $K$. 

Let $E$ and $E'$ be two number fields. Let $M$ be a regular motive over $K$ (with respect to the fixed embedding) with coefficients in $E$ pure of weight $\omega(M)$. Let $M'$ be a regular motive over $K$ with coefficients in $E'$ pure of weight $\omega(M')$. We write $n$ for the rank of $M$ and $n'$ for the rank of $M'$.

We denote by $R(M\otimes M')$ the restriction from $K$ to $\Q$ of the motive $M\otimes M'$. It is a motive pure of weight $\omega:=\omega(M)+\omega(M')$ with Betti realization $M_{B}\otimes M'_{B}\oplus M^{c}_{B}\otimes M'^{c}_{B}$ and De Rham realization $M_{DR}\otimes M'_{DR}\oplus M^{c}_{DR}\otimes M'^{c}_{DR}$.

We denote the Hodge type of $M$ by $\{(p_{i},\omega(M)-p_{i})\mid 1\leq i\leq n\}$ with $p_{1}>\cdots>p_{n}$ and the Hodge type of $M'$ by $\{(r_{j},\omega(M')-r_{j})\mid 1\leq j\leq n\}$ with $r_{1}>r_{2}>\cdots>r_{n'}$. As before, we define $p_{i}^{c}=\omega(M)-p_{n+1-i}$ and $r_{j}^{c}=\omega(M')-r_{n'+1-j}$. They are indices for Hodge type of $M^{c}$ and $M'^{c}$ respectively.

We assume that $R(M\otimes M')$ satisfies Hypothesis \ref{hyp1}, i.e., it has no $(w/2,w/2)$ class. In other words, $p_{a}+r_{b}\neq \cfrac{\omega}{2}$ and then $p_{t}^{c}+r_{u}^{c}\neq \cfrac{\omega}{2}$ for all $1\leq a,t\leq n$, $1\leq b,u\leq n'$.

\begin{rem}
Recall that a necessary condition for a motive $M^{\#}$ to have critical points is that for any integer $p$, the infinity Frobenius $F_{\infty}$ acts as a scalar on $(M^{\#})^{p,p}$. We observe that $R(M\otimes M')^{p,p}=(M\otimes M')^{p,p}\oplus (M^{c}\otimes M'^{c})^{p,p}$ and $F_{\infty}$ interchanges the two factors. Hence $F_{\infty}$ can not be a scalar on $R(M\otimes M')^{p,p}$ unless the latter is zero. Therefore, if the motive $R(M\otimes M')$ has critical points then it has no $(p,p)$ class for any $p$. The converse is also true due to Lemma \ref{critical existence}.
\end{rem}

As in the above section, we take $\{e_{i}\mid 1\leq i \leq n\}$ an $E$-base of $M_{B}$ and define $\{e^{c}_{i}:=F_{\infty}e_{i}\mid 1\leq i\leq n\}$ which is an $E$-base of $M^{c}_{B}$. Similarly, we take $\{f_{j}\mid 1\leq j\leq n'\}$ an $E'$-base of $M'_{B}$ and define $f^{c}_{j}:=F_{\infty}f_{j}$ for $1\leq j\leq n'$.

We also take $\omega_{i}\in M^{p_{i},\omega(M)-p_{i}}$, $(\omega^{c}_{i})\in (M^{c})^{p^{c}_{i},\omega(M)-p^{c}_{i}}$ for $1\leq i\leq n$ as in previous section and $\eta_{j}\in M^{r_{j},\omega(M')-r_{j}}$, $\eta^{c}_{j}\in (M'^{c})^{r^{c}_{j},\omega(M')-r^{c}_{j}}$ for $1\leq j\leq n'$ similarly.

\bigskip

Recall that the motivic periods are defined by \begin{equation}\label{Qi}\tag{P}
F_{\infty}\omega_{i}=Q_{i}\omega^{c}_{n+1-i}, F_{\infty}\mu_{j}=Q'_{j}\mu^{c}_{n'+1-j}
\end{equation}
for $1\leq i\leq n$ and $1\leq j\leq n'$.

The aim of this section is to calculate the Deligne period for $R(M\otimes M')$ in terms of these motivic periods.

\begin{rem}
If we define a paring $(M_{B}\otimes\C)\otimes (M_{B}\otimes\C)\rightarrow \C$ such that $<\omega_{i},\omega_{n+1-i}^{c}>=1$ and $<\omega_{i},\omega_{n+1-j}^{c}>=0$ for $j\neq i$ then $Q_{i}=<\omega_{i},F_{\infty}\omega_{i}>$. 
\end{rem}

\bigskip

We write $M^{\#}=R(M\otimes M')$. We define two sets $A:=\{(a,b) \mid p_{a}+r_{b}> \cfrac{\omega}{2}\}$ and $T:=\{(t,u) \mid p^{c}_{t}+r^{c}_{u}> \cfrac{\omega}{2}\}=\{(t,u) \mid p_{n+1-t}+r_{n'+1-u}< \cfrac{\omega}{2}\}$. 

\bigskip

\begin{rem} 
It is easy to see that \begin{equation}\label{AT}
(t,u)\in T\text{ if and only if }(n+1-t,n'+1-u)\notin A.
\end{equation}
\end{rem}

\bigskip

\begin{prop}\label{deligne period tensor product}
Let $M$, $M'$ be as before. We assume that $M\otimes M'$ has no $(\omega/2,\omega/2)$-class. 
We then have 
\begin{eqnarray} \label{first motivic equation}
c^{+}(R(M\otimes M')) &\sim_{EE'} &\left(\prod_{(t,u)\notin  T}Q_{n+1-t}(M)Q_{n'+1-u}(M')\right)\delta(M\otimes M')\nonumber\\&\sim_{EE'} &\left(\prod_{(t,u)\in  A}Q_{t}(M)Q_{u}(M')\right)\delta(M\otimes M').
\end{eqnarray} 

\end{prop}

\begin{dem}

For simplification of notation, we identify $\omega_{i}\in M_{B}\otimes \C$ with $I_{\infty}(\omega_{i})\in M_{DR}\otimes \C$. Similarly, we identify $\omega^{c}$, $\mu_{j}$, $\mu_{j}^{c}$ with their image under $I_{\infty}$ in the following.

We fixe bases for $M^{+}_{B}$ and $M^{\#}_{DR}/F^{+}(M^{\#})$ now. For $M^{+}_{B}$, we know $$\{ e_{i}\otimes f_{j}+e^{c}_{i}\otimes f^{c}_{j}\mid 1\leq i\leq n, 1\leq j\leq n'\}$$ forms an $EE'$-base. 

To fix a base for $(M^{\#}_{DR} /F^{+}(M^{\#}))\otimes \C=M^{\#}_{DR}\otimes \C /F^{+}(M^{\#})\otimes \C$, we first recall that \begin{equation}\nonumber
I_{\infty}(\bigoplus_{p^{\#}>\frac{\omega}{2}}(M^{\#})^{p^{\#}},\omega-p^{\#})=F^{+}(M^{\#})\otimes \C.
\end{equation}
 
Moreover, we know that $\bigoplus_{p^{\#}>\frac{\omega}{2}}(M^{\#})^{p^{\#},\omega-p^{\#} }$ equals
\begin{equation}\nonumber
(\bigoplus_{a,b\in A}(M)^{p_{a},\omega(M)-p_{a}}\otimes (M')^{r_{b},\omega(M')-r_{b}})\oplus (\bigoplus_{t,u\in T}(M^{c})^{p_{t}^{c},\omega(M)-p_{t}^{c}}\otimes (M'\text{}^{c})^{r_{u}^{c},\omega(M')-r_{u}^{c}}).
\end{equation}

Therefore, the family 
\begin{equation}
\mathcal{B}:= \{ \omega_{a}\otimes \mu_{b}, \omega^{c}_{t}\otimes \mu^{c}_{u} \mod F^{+}(M^{\#})\otimes \C\mid (a,b)\notin A, (t,u)\notin T\}
\end{equation}
 is an $E\otimes\C$ base of $M^{\#}_{DR}\otimes \C /F^{+}(M^{\#})\otimes\C$. This base is not rational but can change to a rational base with a unipotent matrix for change of basis as we have seen before. Therefore we can use this base to calculate the Deligne period.

Note that $F_{\infty}$ is an endomorphism on $M^{\#}_{B}\otimes \C$ and $M^{\#}_{DR}\otimes \C$. For any $\phi\in M^{\#}_{B}\otimes \C$ or $M^{\#}_{DR}\otimes \C$, we write $(1+F_{\infty})\phi:=\phi+F_{\infty}(\phi)$.

If $(a,b)\notin A$ then $(n+1-a,n'+1-b)\in T$ by (\ref{AT}). Along with (\ref{Qi}), we know that $ F_{\infty}(\omega_{a}\otimes \mu_{b})=Q_{a}Q'_{b}\omega_{n+1-a}^{c}\mu_{n'+1-b}^{c} \in F^{+}(M^{\#})\otimes \C$. Similarly, $F_{\infty}( \omega^{c}_{t}\otimes \mu^{c}_{u}) \in F^{+}(M^{\#})\otimes \C$ for all $(t,u)\notin T$. 

We have deduced that
\begin{equation}
 \mathcal{B}=\{(1+F_{\infty})\omega_{a}\otimes \mu_{b},(1+F_{\infty}) \omega^{c}_{t}\otimes \mu^{c}_{u}\mod F^{+}(M^{\#}\otimes\C)\mid (a,b)\notin A, (t,u)\notin T)\}.
 \end{equation}

We take $A_{i,t}$, $B_{j,b}$, $A_{i,t}^{c}$, $B_{j,t}^{c}\in E\otimes \C$ such that $\omega_{a}=\sum\limits_{i=1}^{n}A_{ia}e_{i}$, $\omega_{t}^{c}=\sum\limits_{i=1}^{n}A_{it}^{c}e^{c}_{i}$, $\mu_{b}=\sum\limits_{j=1}^{n'}B_{jb}f_{j}$, $\mu_{u}=\sum\limits_{j=1}^{n'}B^{c}_{ju}f^{c}_{j}$ for $1\leq a,t\leq n$ and $1\leq b,u\leq n'$. 

We then have 
\begin{eqnarray}\nonumber &&(1+F_{\infty})\omega_{a}\mu_{b}=(1+F_{\infty})\sum\limits_{i,j}A_{ia}B_{jb}e_{i}\otimes f_{j}=\sum\limits_{i,j}A_{ia}B_{jb}(e_{i}\otimes f_{j}+e^{c}_{i}\otimes f^{c}_{j})\\
&\text{and }& (1+F_{\infty})\omega^{c}_{t}\omega^{c}_{u}=(1+F_{\infty})\sum\limits_{i,j}A^{c}_{it}B^{c}_{ju}e_{i}^{c}\otimes f_{j}^{c}=\sum\limits_{i,j}A^{c}_{it}B^{c}_{ju}(e_{i}\otimes f_{j} +e_{i}^{c}\otimes f_{j}^{c})\nonumber.
\end{eqnarray}

Up to multiplication by elements in $(EE')^{\times}$, the Deligne period then equals the inverse of the determinant of the matrix 
\begin{equation}\nonumber Mat_{1}:=\left(A_{ia}B_{jb}, A^{c}_{it}B^{c}_{ju}\right)\end{equation} with $1\leq i\leq n, 1\leq j\leq n'$, $(a,b)\notin A$, $(t,u)\notin T$.

\bigskip

By the relation \ref{Qi}, we have $F_{\infty}\omega_{n+1-t}=Q_{n+1-t}\omega^{c}_{t}$. We get 
\begin{equation}
\sum\limits_{i=1}^{n}A_{i,n+1-t}e^{c}_{i}=Q_{n+1-t}\omega^{c}_{t}=Q_{n+1-t}\sum\limits_{i=1}^{n}A_{it}^{c}e^{c}_{i} \nonumber\end{equation}
Therefore, for all $i,j$, we obtain,
\begin{equation}\label{coefficient relation}
A_{it}^{c}=(Q_{n+1-t})^{-1}A_{i,n+1-t}, B_{ju}^{c}=(Q'_{n'+1-u})^{-1}B_{j,n'+1-u} .
\end{equation}

We then deduce that $A^{c}_{it}B^{c}_{ju}=(Q_{n+1-t})^{-1}(Q'_{n'+1-u})^{-1}A_{i,n+1-t}B_{j,n'+1-u}$.

Thus the inverse of the Deligne period:
\begin{equation}\nonumber
c^{+}(R(M\otimes M'))^{-1}\sim_{E(M^{\#})} det(Mat_{1})=\prod\limits_{(t,u)\notin T}((Q_{n+1-t})^{-1}(Q'_{n'+1-u})^{-1})\times det(Mat_{2})
\end{equation} where $Mat_{2}=\left(A_{ia}B_{jb},A_{i,n+1-t,j,n'+1-u}\right)$ with $1\leq i \leq n, 1\leq j\leq n'$, $(a,b)\notin A$ and $(t,u)\notin T$.

\bigskip

Recall that $(t,u)\notin T$ if and only if $(n+1-t,n'+1-u)\in A$. Therefore the index $(n+1-t,n'+1-u)$ above runs over the pairs in $A$. We see that $Mat_{2}=(A_{ia}B_{jb})$ with both $(i,j)$ and $(a,b)$ runs over all the pair in $\{1,2,\cdots,n\}\times \{1,2,\cdots,n'\}$. It is noting but $(A_{ia})_{1\leq i,a \leq n}\otimes (B_{jb})_{1\leq j,b\leq n'}$.

Let us recall the definition of $A_{ia}$. It is the coefficients with respect to the chosen rational bases for the inverse of the comparison isomorphism $M_{B}\otimes \C\rightarrow M_{DR}\otimes \C$. 

Therefore $(A_{ia})_{1\leq i,a \leq n}\otimes (B_{jb})_{1\leq j,b\leq n'}$ is the coefficient matrix of the inverse of the comparison isomorphism \begin{equation}
(M\otimes M')_{B}\otimes \C\rightarrow (M\otimes M')_{DR}\otimes \C.
\end{equation}
 Finally, we get $det((A_{ia})\otimes (B_{jb}))=\delta(M\otimes M')^{-1}$ which terminates the proof.

\end{dem}

\bigskip

 \subsection{Further simplification}

We are going to simplify equation (\ref{first motivic equation}) in this section.

Firstly, it is easy to see that \begin{equation}
\delta(M\otimes M')\sim_{EE'}\delta(M)^{n'}\delta(M')^{n}.
\end{equation}

Secondly, it is already pointed out in Lemma $1.4.2$ of \cite{harrisadjoint} that the set $A$ is a \textit{tableau}, i.e., if $(t,u)\in A$ then for any $t'<t$ and $u'<u$, the pair $(t',u')$ is also in $A$. It is natural to consier some new motivic periods 
\begin{equation}\label{def Qleq}
Q_{(j)}(M):=Q_{1}(M)Q_{2}(M)\cdots Q_{j}(M)\text{ for }1\leq j\leq n
\end{equation} and $Q_{(0)}(M)=1$. We define $Q_{(k)}(M')$ similarly for $0\leq k\leq n'$. We should be able to rewrite equation (\ref{first motivic equation}) in terms of these motivic periods.

In fact, these new motivic periods fit in the automorphic setting in the sense that the automorphic analogue can be defined geometrically. We refer to equation \ref{compare of two periods 1} and equation \ref{compare of two periods 2} for more details.\\

We turn back to the simplification of equation (\ref{first motivic equation}). We need to find out the power for each $Q_{(j)}$. This is a purely combinatorial problem. When the motives are associated to automorphic representations, the precise powers are worked out in \cite{linthesis}. We state a general definition here.

\begin{df}
The sequence $-r_{n'}>-r_{n'-1}>\cdots>-r_{1}$ is split into $n+1$ parts by the numbers $p_{1}-\cfrac{\omega}{2}>p_{2}-\cfrac{\omega}{2}>\cdots>p_{n}-\cfrac{\omega}{2}$. We denote the length of each parts by $sp(i,M;M')$, $0\leq i\leq n$ and call them the \textbf{split indices} for motivic pair.
\end{df}

We can define $sp(j,M';M)$ for $0\leq j\leq n'$ symmetrically. It is easy to see that:
\begin{lem} \label{sum of split index}
The split indices satisfy:
\begin{enumerate}
\item $\sum\limits_{i=0}^{n} sp(i,M;M') =n' =rank(M')$;
\item $sp(i,M;M')=sp(n-i,M^{c};M'\text{}^{c})$ for all $0\leq i\leq n$.
\end{enumerate}
\end{lem}

Recall that $A=\{(a,b) \mid p_{a}+r_{b}> \cfrac{\omega}{2}\}$. For fixed $t$, we have \begin{equation}\#\{u\mid (t,u)\in A \}=\#\{u\mid -r_{u}<p_{t}-\cfrac{\omega}{2}\}.\end{equation}
There is no difficulty to get:
\begin{lem}
For each $1\leq t\leq n$, the cardinal of the set $\{u\mid (t,u)\in A \}$ equals $sp(t,M;M')+sp(t+1,M;M')+\cdots sp(n,M;M')$.
\end{lem}

Therefore, we have 
\begin{eqnarray}
\prod_{(t,u)\in  A}Q_{t}(M)&=&\prod\limits_{t=1}^{n}Q_{t}(M)^{\sum\limits_{j=t}^{n}sp(j,M;M')}= \prod\limits_{j=1}^{n}[\prod\limits_{t=1}^{j}Q_{1}(M)]^{sp(j,M;M')}\nonumber\\
&=& \prod\limits_{j=1}^{n}Q_{(j)}(M)^{sp(j,M;M')}= \prod\limits_{j=0}^{n}Q_{(j)}(M)^{sp(j,M;M')}\nonumber.
\end{eqnarray}

Recall that we have defined $Q_{(0)}(M)=1$.

By the symmetry, we can deduce that
\begin{equation}\label{Q_leq}
\prod_{(t,u)\in A}Q_{t}(M)Q_{u}(M')=\prod\limits_{j=0}^{n}Q_{(j)}(M)^{sp(j,M;M')}\prod\limits_{k=0}^{n'}Q_{(k)}(M')^{sp'(k,M';M)}.\end{equation}

Recall by Lemma \ref{sum of split index} that $\sum\limits_{j=0}^{n} sp(j,M;M') =n'$ and $\sum\limits_{k=0}^{n;} sp(k,M';M) =n$. We know\begin{equation}
\delta(M\otimes M') \sim_{EE'} \delta(M)^{n'}\delta(M')^{n}\sim_{EE'}\prod\limits_{j=0}^{n} \delta(M)^{sp(j,M;M')}\prod\limits_{k=0}^{n} \delta(M')^{sp(k,M';M)}.
\end{equation}

We finally get:
\begin{equation}
c^{+}(R(M\otimes M'))\sim_{EE'} \prod\limits_{j=0}^{n}(Q_{(j)}(M)\delta(M))^{sp(j,M;M')}\prod\limits_{k=0}^{n'}(Q_{(k)}(M')\delta(M))^{sp'(k,M';M)}.
\end{equation}

For the purpose of compatibility with automorphic setting which has different normalization, we define \begin{equation}\label{Delta}
\Delta(M):=(2\pi i)^{\frac{n(n-1)}{2}}\delta(M)\text{ and }\Delta(M'):=(2\pi i)^{\frac{n'(n'-1)}{2}}\delta(M').
\end{equation}

 We set \begin{equation}\label{period Q}
 Q^{(j)}(M)=Q_{(j)}(M)\Delta(M)\text{ and } Q^{(k)}(M')=Q_{(k)}(M')\Delta(M')
 \end{equation} for $0\leq j\leq n$ and $0\leq k\leq n'$.

We conclude that:
\begin{prop}
The Deligne period for $R(M\otimes M')$ satisfies the following formula:
\begin{equation}
c^{+}(R(M\otimes M'))\sim_{EE'}(2\pi i)^{\frac{-nn'(n+n'-2)}{2}}\prod\limits_{j=0}^{n}Q^{(j)}(M)^{sp(j,M;M')}
\prod\limits_{k=0}^{n'}Q^{(k)}(M')^{sp(k,M';M)}.\nonumber
\end{equation}
\end{prop}

\begin{rem}
When $M$ is the motive associated to a conjugate self-dual automorphic representation, the motivic period $Q^{(j)}$ is the same as the motivic period $P_{\leq n+1-s}$ in section $4$ of \cite{harrismotivic}. We shall see this in Proposition \ref{comparison with GH}.
\end{rem}
\bigskip

We remark that for motives restricted from a quadratic imaginary field, the two Deligne periods differ by an element in the coefficient field. Therefore, we may ignore the sign in the original Deligne conjecture. We can now state the Deligne conjecture for tensor product of two motives over $K$.

\begin{conj}\label{Deligne automorphic}

Let $M$ and $M'$ be two pure regular motives over $K$ of rank $n$ and $n'$ respectively. We write $\omega$ for the rank of $M\otimes M'$. We assume that $M\otimes M'$ has no $(\omega/2,\omega/2)$-class.

If $m+\cfrac{n+n'-2}{2}\in \Z$ is a critical point for $M\otimes M'$, then the Deligne conjecture predicts that:
\begin{eqnarray}\label{equation motivic}
&L(m+\frac{n+n'-2}{2},R(M\otimes M')) \sim_{EE'}&\\&
(2\pi i)^{nn'm}\prod\limits_{j=0}^{n}Q^{(j)}(M)^{sp(j,M;M')}
\prod\limits_{k=0}^{n'}Q^{(k)}(M')^{sp(k,M';M)}.\nonumber&
\end{eqnarray}

\end{conj}

\section{The automorphic variant}
\bigskip
 \subsection{Motives associated to automorphic representations}

Let $\Pi=\Pi_{f}\otimes \Pi_{\infty}$ be a cuspidal representation of $GL_{n}(\AK)$. We denote the infinity type of $\Pi$ by $(z^{a_{i}}\overline{z}^{b_{i}})_{1\leq i\leq n}$.

We assume that $\Pi$ is \textbf{regular} which means that its infinity type satisfies $a_{i}\neq a_{j}$ for all $1\leq i<j\leq n$. We assume also $\Pi$ is \textbf{algebraic} which means that its infinity type satisfies $a_{i},b_{i}\in\Z+\cfrac{n-1}{2}$ for all $1\leq i\leq n$. By Lemma $4.9$ of \cite{clozelaa}), we know that there exists an integer $w_{\Pi}$ such that $a_{i}+b_{i}=-w_{\Pi}$ for all $i$.

We denote by $V$ the representation space of $\Pi_{f}$. For $\sigma\in Aut(\C)$, we define $\Pi_{f}^{\sigma}$, another $GL_{n}(\AK\text{}_{,f})$-representation, to be $V\otimes_{\C,\sigma}\C$. Let $\Q(\Pi)$ be the subfield of $\C$ fixed by $\{\sigma\in Aut(\C) \mid \Pi_{f}^{\sigma} \cong \Pi_{f}\}$. We call it the \textbf{rationality field} of $\Pi$.

For $E$ a number field, $G$ a group and $V$ a $G$-representation over $\C$, we say $V$ has an \textbf{$E$-rational structure} if there exists a $E$-vector space $V_{E}$ endowed with an action of $G$ such that $V=V_{E}\otimes_{E}\C$ as a representation of $G$. We call $V_{E}$ an $E$-rational structure of $V$. We have the following result (c.f. \cite{clozelaa} Theorem $3.13$):

\begin{thm}
For $\Pi$ a regular algebraic cuspidal representation of $GL_{n}(\AK)$, $\Q(\Pi)$ is a number field. Moreover, $\Pi_{f}$ has a $\Q(\Pi)$-rational structure. For all $\sigma\in Aut(\C)$, $\Pi_{f}^{\sigma}$ is the finite part of a cuspidal representation of $GL_{n}(\AK)$ which is unique by the strong multiplicity one theorem, denoted by $\Pi^{\sigma}$.
\end{thm}

It is conjectured that such a $\Pi$ is attached to a motive with coefficients in a finite extension of $\Q(\Pi)$:
\begin{conj}\label{conjmotive}(Conjecture $4.5$ and paragraph $4.3.3$ of \cite{clozelaa})

Let $\Pi$ be a regular algebraic cuspidal representation of $GL_{n}(\AK)$ and $\Q(\Pi)$ its rationality field. There exists $E$ a finite extension of $\Q(\Pi)$ and $M$ a regular motive of rank $n$ over $K$ with coefficients in $E$ such that
\begin{equation}\nonumber 
L(s,M,\sigma)=L(s+\cfrac{1-n}{2},\Pi^{\sigma})\end{equation}
for all $\sigma: E\hookrightarrow \C$.

Moreover, if the infinity type of $\Pi$ is $(z^{a_{i}}\overline{z}^{-\omega(\Pi)-a_{i}})_{1\leq i\leq n}$, then the Hodge type of $M$ is $\{(-a_{i}+\cfrac{n-1}{2},\omega(\Pi)-a_{i}+\cfrac{n-1}{2})\mid 1\leq i\leq n\}$.

In particular, this motive is pure of weight $w_{\Pi}+n-1$.

\end{conj}

This is a part of the Langlands program. In light of this, we want to associate a motive, or more precisely, a geometric object to certain representations. Unfortunately there is no hermitian symmetric domain associated to  $GL_{n}$ for $n>2$. We can not use theory of Shimura variety directly. We need to pass to unitary groups via base change theory.

\bigskip
\subsection{Base change and Langlands functoriality}

Let $G$ and $G'$ be two connected quasi-split reductive algebraic groups over $\Q$. Put $\widehat{G}$ the complex dual group of $G$. The Galois group $G_{\Q}=Gal(\overline{\Q}/\Q)$ acts on $\widehat{G}$. We define the \textbf{$L$-group} of $G$ by $\text{}^{L}G:=\widehat{G} \rtimes G_{\Q}$ and we define $\text{}^{L}G'$ similarly. A group homomorphism between two $L$-groups $\text{}^{L}G\rightarrow \text{}^{L}G'$ is called an \textbf{$L$-morphism} if it is continuous, its restriction to $\widehat{G}$ is analytic and it is compatible with the projections of $\text{}^{L}G$ and $\text{}^{L}G'$ to $G_{\Q}$. If there exists an $L$-morphism $\text{}^{L}G\rightarrow \text{}^{L}G'$, the \textbf{Langlands principal of functoriality} predicts a correspondence between automorphic representations of $G(\AQ)$ and automorphic representations of $G'(\AQ)$ (c.f. section $26$ of \cite{arthurtraceformula}). More precisely, we wish to associate an $L$-packet of automorphic representations of $G(\AQ)$ to that of $G'(\AQ)$.

Locally, we can specify this correspondence for unramified representations. Let $p$ be a finite place of $\Q$ such that $G$ is unramified at $p$. We fix $\Gamma_{p}$ a maximal compact hyperspecial subgroup of $G_{p}:=G(\Q_{p})$. By definition, for $\pi_{p}$ an admissible representation of $G_{p}$, we say $\pi_{p}$ is \textbf{unramified} (with respect to $\Gamma_{p}$) if it is irreducible and $dim \pi_{p}^{\Gamma_{p}} > 0$. One can show that $\pi_{p}^{\Gamma_{p}}$ is actually one dimensional since $\pi_{p}$ is irreducible.

We set $H_{p}:=\mathcal{H}(G_{p},\Gamma_{p})$ the Hecke algebra consisting of compactly supported continuous functions from $G_{p}$ to $\C$ which are $\Gamma_{p}$ invariants at both side. We know $H_{p}$ acts on $\pi_{p}$ and preserves $\pi_{p}^{\Gamma_{p}}$ (c.f. \cite{minguez}). Since $\pi_{p}^{\Gamma_{p}}$ is one-dimensional, every element in $H_{p}$ acts as a multiplication by a scalar on it. In other words, $\pi_{p}$ thus determines a character of $H_{p}$. This gives a map from the set of unramified representations of $G_{p}$ to the set of characters of $H_{p}$ which is in fact a bijection (c.f. section $2.6$ of \cite{minguez}).

We can moreover describe the structure of $H_{p}$ in a simpler way. Let $T_{p}$ be a maximal torus of $G_{p}$ contained in a Borel subgroup of $G_{p}$. We denote by $X_{*}(T_{p})$ the set of cocharacters of $T_{p}$ defined over $\Q_{p}$. The Satake transform identifies the Hecke algebra $H_{p}$ with the polynomial ring $\C[X_{*}(T_{p})]^{W_{p}}$ where $W_{p}$ is the Weyl group of $G_{p}$ (c.f. section $1.2.4$ of \cite{harristakagi}).

Let $G'$ be a connected quasi-split reductive group which is also unramified at $p$. We can define $\Gamma'_{p}$, $H'_{p}:=\mathcal{H}(G'_{p},\Gamma'_{p})$ and $T'_{p}$ similarly. An $L$-morphism $\text{}^{L}G\rightarrow \text{}^{L}G'$ induces a morphism $\widehat{T_{p}}\rightarrow \widehat{T'_{p}}$ and hence a map $T'_{p}\rightarrow T_{p}$. The latter gives a morphism from $\C[X_{*}(T'_{p})]^{W'_{p}}$ to $\C[X_{*}(T_{p})]^{W_{p}}$ and thus a morphism from $H'_{p}$ to $H_{p}$. Its dual hence gives a map from the set of unramified representations of $G_{p}$ to that of $G'_{p}$. This is the local Langlands principal of functoriality for unramified representations.

\bigskip

In this article, we are interested in a particular case of the Langlands functoriality: the cyclic base change. Let $F$ be a cyclic extension of $\Q$, for example, a quadratic imaginary field. Let $G$ be a connected quasi-split reductive group over $\Q$. Put $G'=Res_{F/\Q}G_{F}$. We know $\widehat{G'}$ equals $\widehat{G}^{[F:\Q]}$. The diagonal embedding is then a natural $L$-morphism $\text{}^{L}G \rightarrow \text{}^{L}G'$. The corresponding functoriality is called the (global) \textbf{base change}.

More precisely, let $p$ be a finite place of $\Q$. The above $L$-morphism gives a map from the set of unramified representations of $G(\Q_{p})$ to that of $G'(\Q_{p})$ where the latter is isomorphic to $G(F_{p})\cong \bigotimes\limits_{v|p} G(F_{v})$. By the tensor product theorem, all the unramified representation of $G(F_{p})$ can be written uniquely as tensor product of unramified representations of $G(F_{v})$ where $v$ runs over the places of $F$ above $p$. We fix $v$ a place of $F$ above $p$ and we then get a map from the set of unramified representation of $G(\Q_{p})$ to that of $G(F_{v})$. For an unramified representation of $G(\Q_{v})$, we call the image its (local) \textbf{base change} with respect to $F_{v}/\Q_{p}$.

Let $\pi$ be an admissible irreducible representation of $G(\AQ)$. We say $\Pi$, a representation of $G(\AF)$, is a \textbf{(weak) base change} of $\pi$ if for all $v$, a finite place of $\Q$, such that $\pi$ is unramified at $v$ and all $w$, a place of $F$ over $v$, $\Pi_{w}$ is the base change of $\pi_{v}$. In this case, we say $\Pi$ \textbf{descends to $\pi$} by base change. Moreover, if $p$ is unramified for both $\pi$ and $\Pi$, it is easy to see that $\Pi_{p}$ and $\pi_{p}$ have the same $L$-factor.

\begin{rem}
The group $Aut(F)$ acts on $G(\AF)$. This induces an action of $Aut(F)$ on automorphic representations of $G(\AF)$. We denote this action by $\Pi^{\sigma}$ for $\sigma\in Aut(F)$ and $\Pi$ an automorphic representation of $G(\AF)$. It is easy to see that if $\Pi$ is a base change of $\pi$, then $\Pi^{\sigma}$ is one as well. So if we have the strong multiplicity one theorem for $G(\AF)$, we can conclude that every representation in the image of base change is $Aut(F)$-stable.

\end{rem}

\begin{ex}\textbf{Base change for $GL_{1}$}

Now let us give an example of base change. Let $F/L$ be a cyclic extension of numbers fields. Let $\sigma$ be a generator of $Gal(F/L)$. The automorphic representations of $GL_{1}$ over a number field are nothing but Hecke characters. Let $\eta$ be a Hecke character of $L$. The base change of $\eta$ to $GL_{1}(\AF)$ in this case is just $\eta\circ N_{\AF/\AL}$.

On the other hand, as discussed above, if $\chi$ is a Hecke character of $\AF$, then a necessary condition for it to descend is $\Pi =\Pi^{\sigma}$ for all $\sigma\in Gal(F/L)$. We shall see that this is also sufficient.

\begin{thm}\label{bccharacter}
Let $F/L$ be a cyclic extension of number fields and $\chi$ be a Hecke character of $F$. If $\chi=\chi^{\sigma}$ for all $\sigma\in Gal(F/L)$, then there exists $\eta$, a Hecke character of $L$, such that $\chi=\eta\circ N_{\AF/\AL}$.
Moreover, if $\chi$ is unramified at some place $v$ of $L$, we can choose $\eta$ to be unramified at places of $K$ dividing by $v$.
\end{thm}
\begin{proof}
We define at first $\eta_{0}:N_{\AF/\AL}(\AF)^{\times}\rightarrow \C$ as follows:
for any $w\in N_{\AF/\AL}(\AF)^{\times} $, take $z\in \AF^{\times}$ such that $w=N_{\AF/\AL}(z)$ and we define $\eta_{0}(w)=\chi(z)$. 

This does not depend on the choice of $z$. In fact, if $z'\in\AF^{\times}$ such that $N_{\AF/\AL}(z')$ also equals $w=N_{\AF/\AL}(z)$, then by Hilbert's theorem $90$ which says $H^{1}(Gal(F/L), \AF^{\times})=1$, there exists $t\in \AF^{\times}$ such that $z'=\cfrac{\sigma(t)}{t}z$ for some $\sigma\in Gal(F/L)$. Hence $\chi(z')=\cfrac{\chi^{\sigma}(t)}{\chi(t)}\chi(z)=\chi(z)$. Therefore $\eta_{0}(w)$ is well defined.
One can verify that $\eta_{0}$ is a continuous character.

By Hasse norm theorem, $N_{\AF/\AL}(\AF)^{\times}\cap L^{\times} =N_{\AF/\AL}(F)^{\times}$ (this is a direct corollary of Hilbert's theorem $90$ on $L^{\times}\backslash\AL^{\times}$ that $H^{1}(Gal(F/L), L^{\times}\backslash\AL^{\times})=\{1\}$), we know $\eta_{0}$ is trivial on $N_{\AF/\AL}(\AF)^{\times}\cap L^{\times}$, and hence factors through $(L^{\times}\cap N_{\AF/\AL}(\AF^{\times}))\backslash\N_{\AF/\AL}(\AF^{\times})$. The latter is a finite index open subgroup of $L^{\times}\backslash \AL^{\times}$ by the class field theory. We can thus extend $\eta_{0}$ to a Hecke character of $L$ as we want.
\end{proof}

\end{ex}

\bigskip

We now consider the base change for unitary groups and similitude unitary groups with respect to $K/\Q$. 

For each integer $0\leq s\leq n$, there exists a Hermitian space of dimension $n$ over $K$ with infinity sign $(n-s,s)$ such that the associated unitary group $U_{s}$ over $\Q$ is quasi-split at all finite place except for at most one finite space split in $K$ (c.f. section $2$ of \cite{clozelIHES} or section $1.2$ of \cite{harrislabesse}). 

We remark that $U_{s}(\AK) \cong GL_{n}(\AK)$. We can regard the regular algebraic cuspidal representation $\Pi$ as a representation of $U_{s}(\AK)$. We denote by $\Pi^{\vee}$ the contragredient representation of $\Pi$. We say $\Pi$ is \textbf{conjugate self-dual} if $\Pi^{c}\cong \Pi^{\vee}$. If $\Pi$ descends to $U_{s}(\AQ)$ then by the previous discussion we know that it is stable under $Aut(K)$-action. This is equivalent to say that $\Pi$ is conjugate self-dual.\\

If $n$ is odd, then we can take $U_{s}$ quasi-split at each finite place. In this case, the conjugate self-dual condition is also sufficient. 

If $n$ is even and $s$ has the same parity as $n/2$, it is in the same situation as the $n$ odd case. But if $n$ is even and $s$ does not have the same parity as $n/2$, the unitary group $U_{s}$ can not be quasi-split everywhere. In this case, we have to add a local condition (see the hypothesis below) so that $\Pi$ descends to $U_{s}(\AQ)$. Since we will need that $\Pi$ descends to $U_{s}(\AQ)$ for each $s$, we postulate the following hypothesis whenever $n$ is even:

\begin{hyp}\label{local condition}
The representation $\Pi$ is a discrete series representation at a finite place of $\Q$ which is split in $K$.
\end{hyp}

\begin{prop}(Theorem $3.1.2$, $3.1.3$ of \cite{harrislabesse})
Let $\Pi$ be a regular algebraic cuspidal representation of $GL_{n}(\AK)$ which is moreover conjugate self-dual. If $n$ is even, we assume that $\Pi$ satisfies Hypothesis \ref{local condition}. 

For each $0\leq s\leq n$, let $U_{s}$ be the unitary group as before. We know $\Pi$ descends to a cuspidal representation of $U_{s}(\AQ)$ for any $s$.
\end{prop}
\bigskip

We denote $GU_{s}$ the rational similitude unitary group associated to $U_{s}$. We have $GU_{s}(\AK)\cong GL_{n}(\AK)\otimes \AK^{\times}$.

We write $\xi_{\Pi}$ for the central character of $\Pi$. It is conjugate self-dual since $\Pi$ is. One can show that there exists $\xi$, an algebraic Hecke character of $\AK$, such that \begin{equation}\label{xi}
\cfrac{\xi(\overline{z})}{\xi(z)}=\xi_{\Pi}(z)^{-1}.
\end{equation} 
This is again due to Hilbert $90$. For detailed proof, we refer to Lemma $2.2$ of \cite{clozelramanujan} or Lemma $2.3.1$ of \cite{linthesis}. Hence, by a similar argument as Theorem $VI.2.9$ in \cite{harristaylor}, we can deduce that that $\Pi^{\vee}\otimes \xi$ descends to a cuspidal representation of $GU_{s}(\AQ)$, denoted by $\pi_{s}$.

We are about to consider cohomology space associated to $\pi_{s}$. We need to assume that $\Pi$ is \textbf{cohomological} so that $\pi_{s}$ is also cohomological.

Let $G_{\infty}$ be the group of real points of $Res_{K/\Q}GL_{n}$. Recall that $\Pi$ is cohomological if there exists $W$ an algebraic finite-dimensional representation of $G_{\infty}$ such that $H^{*}(\mathcal{G}_{\infty},K_{\infty}; \Pi_{\infty}\otimes W)\neq 0$ where $\mathcal{G}_{\infty}=Lie(G_{\infty})$ and $K_{\infty}$ is the product of a maximal compact group of $G_{\infty}$ and the center of $G_{\infty}$. We remark that a cohomological representation is automatically algebraic since its infinity type is determined by this algebraic finite-dimensional representation.

We write $W(\pi_{\infty})$ for the finite dimensional representation associated to $\pi_{s}$ by the cohomological property.

\bigskip
\subsection{Automorphic periods}

We will summarize the construction of automorphic periods in \cite{harris97} in this section. Firstly, we construct a Shimura datum on $GU_{s}$.

We define $X_{s}$ to be the $GU_{s}(\R)$ conjugate class of \begin{eqnarray}
\nonumber h_{s}: \mathbb{S}(\C)=\C^{\times}\times \C^{\times}&\rightarrow &GU_{s}(\C)\cong GL_{n}(\C)\times \C^{\times}\\\nonumber
(z,\overline{z})&\mapsto & (\begin{pmatrix}
zI_{r} & 0\\\nonumber
0 & \overline{z}I_{s}
\end{pmatrix},
z\overline{z}).
\end{eqnarray}

 We know that $(GU_{s},X_{s})$ is a Shimura datum and we denote by $Sh(GU_{s})$ the associated Shimura variety. The finite dimensional representation $W(\pi_{\infty})$ defines a complex local system $\mathcal{W}(\pi_{\infty})$ and $l$-adic local system $\mathcal{W}(\pi_{\infty})_{l}$ on $Sh(GU_{s})$. 
 
As shown in \cite{harris97}, the cohomology group $\overline{H}^{rs}(Sh(GU_{s}), \mathcal{W}(\pi_{\infty}))$ defined in section $2.2$ of \cite{harris97} is naturally endowed with a De Rham rational structure and a Betti rational structure over $K$ (c.f. Proposition $2.2.7$ of \textit{loc.cit}). The cohomology group $\overline{H}^{rs}(Sh(GU_{s}), \mathcal{W}(\pi_{\infty})_{l})$ is endowed with an $l$-adic structure. Moreover, $\pi_{f}$ contributes non trivially to these cohomology groups, i.e. $\overline{H}^{rs}(Sh(GU_{s}),*)[\pi_{f}]:=Hom_{G(\AK\text{}_{f})}(\pi_{f},\overline{H}^{rs}(Sh(GU_{s}),*)\neq 0$ for $*= \mathcal{W}(\pi_{\infty})$ or $\mathcal{W}(\pi_{\infty})_{l}$.

One direct consequence is that the rationality field of $\pi_{f}$ is a number field (see section $2.6$ of \cite{harris97}). One can then realize $\pi_{f}$ over $E(\pi)$, a finite extension of its rationality field, which is still a number field. We take $E(\Pi)$ a number field which contains the $E(\pi)$ for all $s$. We also assume that $E(\Pi)$ contains $K$ for simplicity.

\medskip

One can show that $\overline{H}^{rs}(Sh(GU_{s}),*)[\pi_{f}]$ for $*= \mathcal{W}(\pi_{\infty})$ or $\mathcal{W}(\pi_{\infty})_{l}$ form a motive with coefficients in $E(\Pi)$ (c.f. Proposition $2.7.10$ of \textit{loc.cit}). We denote it by $M_{s}(\Pi,\xi)$.

Since $\Pi$ is conjugate self-dual, there exists $<.>$ a non degenerate bilinear form on $\overline{H}^{rs}(Sh(GU_{s}),\mathcal{W}(\pi_{\infty}))[\pi_{f}]$ normalized as in section $2.6.8$ of \cite{harris97}.

\bigskip

The Hodge decomposition of $\overline{H}^{rs}(Sh(GU_{s}),\mathcal{W}(\pi_{\infty}))[\pi_{f}]$ is given by the coherent cohomology. More precisely, let $Sh(GU_{s})\hookrightarrow \widetilde{Sh}(GU_{s})$ be a smooth toroidal compactification. Let $\pi'_{\infty}$ be any discrete series representation of $GU_{s}(\R)$ with base change $\Pi^{\vee}_{\infty}$. It is then cohomological and we denote $\widetilde{\mathcal{E}}(\pi'_{\infty})$ the coherent automorphic vector bundle attached to the finite dimensional representation associated to $\pi'_{\infty}$. We have \begin{equation}\nonumber 
\overline{H}^{rs}(Sh(GU_{s}),\mathcal{W}(\pi_{\infty}))[\pi_{f}]=\bigoplus\limits_{\pi_{\infty}'}\widetilde{H}^{q(\pi_{\infty}')}(\widetilde{Sh}(GU_{s}),\widetilde{\mathcal{E}}(\pi'_{\infty}))[\pi_{f}]
\end{equation} where $\widetilde{H}$ indicates the coherent cohomology and $q(\pi_{\infty}')$ is an integer depends on $\pi_{\infty}'$. 

Among these $\pi_{\infty}'$, there exists a holomorphic representation $\pi'_{\infty}$ such that $q(\pi_{\infty}')=0$. We fix this $\pi_{\infty}'$ and choose a $K$-rational element \begin{equation}\nonumber
0\neq \beta_{s}\in \widetilde{H}^{q(\pi_{\infty}')=0}(\widetilde{Sh}(GU_{s}),\widetilde{\mathcal{E}}(\pi'_{\infty}))[\pi_{f}].
\end{equation}

At last, we take the integer $C$ such that $\xi_{\infty}(t)=t^{C}$ for $t\in \R^{+}$. We define the \textbf{automorphic period} of $(\Pi,\xi)$ to be the normalized Petersson inner product:
\begin{equation}\nonumber
P^{(s)}(\Pi,\xi):=(2\pi)^{-C}<\beta_{s},\beta_{s}>.
\end{equation} It is a non zero complex number. By the following proposition, we see that $P^{(s)}(\Pi,\xi)$ does not depend on the choice of $\beta_{s}$ modulo $E(\pi)^{\times}$ and thus well defined modulo $E(\pi)^{\times}$.
\begin{prop}(Proposition $3.19$ in \cite{harrisimrn})

Let $\beta_{s}'\in \widetilde{H}^{q(\pi_{\infty}')=0}(\widetilde{Sh}(GU_{s}),\widetilde{\mathcal{E}}(\pi'_{\infty}))[\pi_{f}]$ be another $K$-rational element. We have $\cfrac{<\beta_{s},\beta_{s}'>}{<\beta_{s},\beta_{s}>}\in E(\pi)$. 

Consequently, $<\beta_{s},\beta_{s}>\sim_{E(\pi)}<\beta_{s}',\beta_{s}'>$.
\end{prop}

\begin{rem}
We may furthermore choose $\beta_{s}$ such that $<\beta_{s},\beta_{s}>$ is equivariant under action of $G_{K}$.
\end{rem}

Actually, this period $P^{(s)}(\Pi,\xi)$ is also independent of the choice of $\xi$. This is a corollary to Theorem $3.5.13$ of \cite{harris97}. We can also deduce it from our main theorem in the next section. At the moment we just fix a $\xi$ and define the ($s$-th) \textbf{automorphic period} of $\Pi$ by $P^{(s)}(\Pi):=P^{(s)}(\Pi,\xi)$.

\bigskip
\subsection{An automorphic variant of the Deligne conjecture}
We assume that the motive associated to $\Pi$ exists and denote it by $M:=M(\Pi)$. Similarly, for the Hecke character $\xi$, we denote by $M(\xi)$ the associated motive.

Recall that $\pi_{s}$ has base change $\Pi^{\vee}\otimes \xi$. It is expected, if we admit the Tate conjecture, that:
\begin{equation}\label{motive relation}
M_{s}(\Pi,\xi) \cong \Lambda^{n-s} (M(\Pi)^{\vee}) \otimes M(\xi)
\end{equation}
up to twist by a Tate motive.

We have taken $\beta_{s}$ in the bottom stage of the Hodge filtration. Therefore, we should have:
\begin{equation}\label{compare of two periods 1}
P^{(s)}(\Pi) \sim_{E(\Pi)} (Q_{1}(M^{\vee})Q_{2}(M^{\vee})\cdots Q_{n-s}(M^{\vee})) \times Q_{1}(M(\xi)).
\end{equation}

Recall that $\cfrac{\xi^{c}}{\xi}=\xi_{\Pi}^{-1}$ by equation (\ref{xi}). The right hand side in fact equals $Q^{(s)}(M)$. We shall give a lemma first and then prove this property.
\begin{lem}(Lemma $1.2.7$ of \cite{harrisadjoint})
Under the condition that $\Pi$ is conjugate self-dual, we have:
\begin{equation}
\delta(M)^{-2} (2\pi i)^{n(1-n)} \sim_{E} \prod\limits_{1\leq i\leq n}Q_{i}.
\end{equation}
\end{lem}

\begin{dem}
Recall by Lemma \ref{delta c} that 
\begin{equation}\label{first relation}
\delta(M^{c})\sim_{E}(\prod\limits_{1\leq i\leq n}Q_{i})\delta(M).
\end{equation}

On one hand, the comparison isomorphism for $M^{\vee}$ is the inverse of the dual of the comparison isomorphism for $M$. Hence $\delta(M^{\vee}) \sim_E \delta(M)^{-1}$. This is true for all motives.

On the other hand, since $\Pi$ is conjugate self-dual, we have:
$M^{c} \cong M^{\vee}(1-n).$ This implies that
\begin{equation}
\delta(M^{c}) \sim_{E} \delta(M^{\vee})(2\pi i)^{n(1-n)}.
\end{equation}

We then deduce that \begin{equation}
\delta(M^{c})\sim_{E} \delta(M)^{-1} (2\pi i)^{n(1-n)}.
\end{equation}

We compare this with equation (\ref{first relation}) and get:
\begin{equation}
\delta(M)^{-1} (2\pi i)^{n(1-n)} \sim_{E} (\prod\limits_{1\leq i\leq n}Q_{i})\delta(M).
\end{equation}
The lemma then follows.
\end{dem}

We can now prove the following:
\begin{prop}\label{comparison with GH}
Let $M=M(\Pi)$. Let $\xi$ be a Hecke character of $\AK$ such that $\cfrac{\xi^{c}}{\xi}=\xi_{\Pi}^{-1}$. We then have:
\begin{equation}\label{compare of two periods 2}
Q^{(s)}(M)\sim_{E(\Pi);K}(Q_{1}(M^{\vee})Q_{2}(M^{\vee})\cdots Q_{n-s}(M^{\vee})) \times Q_{1}(M(\xi))
\end{equation}
\end{prop}
\begin{dem}
Recall that 
\begin{equation}
Q^{(s)}(M)=Q_{(s)}(M)\Delta(M)=(Q_{1}(M)Q_{2}(M)\cdots Q_{s}(M))(2\pi i)^{\frac{n(n-1)}{2}}\delta(M)
\end{equation} by equation (\ref{Delta}) and equation (\ref{period Q}).

Moreover, by Lemma \ref{period for Mc} , we have:
\begin{equation}
Q_{1}(M^{\vee})Q_{2}(M^{\vee})\cdots Q_{n-s}(M^{\vee})  \sim_{E(\Pi)} Q_{s+1}(M)^{-1}Q_{s+2}(M)^{-1}\cdots Q_{n}(M)^{-1}. \label{first reverse}
\end{equation}

By the previous lemma, we have:
\begin{eqnarray}\nonumber
&&(Q_{s+1}(M)^{-1}Q_{s+2}(M)^{-1}\cdots Q_{n}(M)^{-1}) \times Q_{1}(M(\xi))\\\nonumber
&\sim_{E}& Q_{1}(M)Q_{2}(M)\cdots Q_{s}(M) (2\pi i)^{n(n-1)}\delta(M)^{2}  \times Q_{1}(M(\xi))\\\nonumber
&\sim_{E}& (Q_{(s)}(M) (2\pi i)^{\frac{n(n-1)}{2}}\delta(M))(2\pi i)^{\frac{n(n-1)}{2}}\delta(M)  \times Q_{1}(M(\xi))\\\nonumber
&\sim_{E}& Q^{(s)}(M)(2\pi i)^{\frac{n(n-1)}{2}}\delta(M)  \times Q_{1}(M(\xi)).
\end{eqnarray}


Therefore it remains to show that $Q_{1}(M(\xi))^{-1}\sim_{E(\Pi);K} (2\pi i)^{\frac{n(n-1)}{2}}\delta(M)$.

Since $det(M(\Pi))\cong M(\xi_{\Pi})(-\cfrac{n(n-1)}{2})$, by equation (\ref{Tate lemma}), we have 
\begin{equation}
\delta(M)\sim_{E(\Pi);K }\delta(det(M(\Pi)))\sim_{E(\Pi);K}\delta(M(\xi_{\Pi}))(2\pi i)^{-\frac{n(n-1)}{2}}.
\end{equation} 

At last, we recall that $\cfrac{\xi^{c}}{\xi}=\xi_{\Pi}^{-1}$. We can show $\delta(M(\xi_{\Pi})) \sim_{E(\Pi)}Q_{1}(M(\xi))^{-1}$ with the help of CM periods and Blasius's result on special values of $L$-functions for Hecke characters. We refer to the appendix of \cite{harrisappendix} for the notation and section $6.4$ of \cite{linthesis} for the proof.

\end{dem}

We now consider critical values of automorphic $L$-functions for tensor product of two cuspidal representations. As predicted by Conjecture \ref{Deligne automorphic}, they should be able to interpreted in terms of automorphic periods.

More precisely, let $n'$ be a positive integer and $\Pi'$ be an conjugate self-dual cohomological cuspidal regular representation of $GL_{n'}(\AK)$. If $n'$ is even, we assume that $\Pi'$ satisfies Hypothesis \ref{local condition}.

We denote the infinity type of $\Pi$ (resp. $\Pi'$) by $(z^{a_{i}}\overline{z}^{-\omega(\Pi)-a_{i}})_{1\leq i\leq n}$ (resp. $(z^{b_{j}}\overline{z}^{-\omega(\Pi)-b_{j}})_{1\leq j\leq n'}$) with $a_{1}>a_{2}>\cdots>a_{n}$ (resp. $b_{1}>b_{2}>\cdots>b_{n'}$). We suppose $\Pi\times \Pi'$ is critical, i.e. $a_{i} +b_{j}\neq -\cfrac{\omega(\Pi)+\omega(\Pi')}{2}$ for all $i,j$. Otherwise $\Pi\times \Pi'$ does not have critical values.

Let $m\in \Z+\cfrac{n+n'}{2}$. We say $m$ is \textbf{critical} for $\Pi\times \Pi'$ if $m+\cfrac{n+n'-2}{2}$ is critical for $R(M(\Pi)\otimes M(\Pi'))$. Recall that critical points are determined by the Hodge type as in Lemma \ref{critical existence}. Moreover, the Hodge type of $R(M(\Pi)\otimes M(\Pi'))$ is determined by the infinity type of $\Pi$ and $\Pi'$. Therefore, even if $M(\Pi)$ or $M(\Pi')$ does not exist, we may still define critical points for $\Pi\times \Pi'$ in this way. 

We can also give an explicit criteria on critical points by Lemma \ref{critical existence}. We recall that the motive $M(\Pi)$ associated to $\Pi$, if it exists, should have Hodge type $(-a_{n+1-i}+\frac{n-1}{2}, a_{n+1-i}+\omega(\Pi)+\frac{n-1}{2})_{1\leq i\leq n}$ by Conjecture \ref{conjmotive}.

\begin{lem} 
Let $m\in \Z+\cfrac{n+n'}{2}$. It is critical for $\Pi\times \Pi'$ if and only if for any $1\leq i\leq n$ and any $1\leq j\leq n'$, if $a_{i}+b_{j}>-\cfrac{\omega(\Pi)+\omega(\Pi')}{2}$ then $-a_{i}-b_{j}<m<a_{i}+b_{j}+\omega(\Pi)+\omega(\Pi')+1$; otherwise $a_{i}+b_{j}+\omega(\Pi)+\omega(\Pi')<m<-a_{i}-b_{j}+1$.
In particular, critical point always exists if $\Pi\times \Pi'$ is critical.
\end{lem}

\bigskip

We now define the split index for automorphic pairs. We split the sequence $b_{1}>b_{2}>\cdots >b_{n'}$ by the numbers $-a_{n}-\cfrac{\omega(\Pi)+\omega(\Pi')}{2}>-a_{n-1}-\cfrac{\omega(\Pi)+\omega(\Pi')}{2}>\cdots>-a_{1}-\cfrac{\omega(\Pi)+\omega(\Pi')}{2}$. We denote the length of each part by $sp(j,\Pi;\Pi')$ for $0\leq j\leq n$ and call them the \textbf{split indices} for automorphic pairs. We may define $sp(k,\Pi';\Pi)$ for $0\leq k\leq n'$ symmetrically.
 It is easy to see that $sp(j,M(\Pi);M(\Pi'))=sp(j,\Pi;\Pi')$ for all $0\leq j\leq n$ and $sp(k,M(\Pi');M(\Pi))=sp(k,\Pi';\Pi)$ for all $0\leq k\leq n'$.

Finally, since $L(m,\Pi\times\Pi')=L(m+\cfrac{n+n'-2}{2},R(M(\Pi)\otimes M(\Pi')))$, Conjecture \ref{Deligne automorphic} then predicts an automorphic variant of the Deligne conjecture:

\begin{conj}\label{Deligne automorphic 2}
Let $\Pi$ and $\Pi'$ be as before. If $m\in \Z+\cfrac{n+n'}{2}$ be critical for $\Pi\times \Pi'$ then we have:
\begin{equation}\label{equation automorphic}
L(m,\Pi\times \Pi') \sim_{E(\Pi)E(\Pi');K}(2\pi i)^{nn'm}\prod\limits_{j=0}^{n}P^{(j)}(\Pi)^{sp(j,\Pi;\Pi')}
\prod\limits_{k=0}^{n'}P^{(k)}(\Pi')^{sp(k,\Pi';\Pi)}.\nonumber
\end{equation}
\end{conj}
 
This conjecture is a purely automorphic statement. Moreover, it is proved by automorphic methods in many cases. We state some known results here.

\begin{thm}\label{main theorem}
We assume that both $\Pi$ and $\Pi'$ are very regular, i.e. the numbers $a_{i}-a_{i+1}\geq 3$ and $b_{j}-b_{j+1}\geq 3$ for all $1 \leq i\leq n-1$ and $1\leq j\leq n'-1$.

We may assume that $n\geq n'$. Conjecture \ref{Deligne automorphic 2} is true in the following cases:
\begin{enumerate}
\item $n'=1$, $\Pi'$ need not to be conjugate self-dual;
\item $n>n'$, $n$ and $n'$ have different parity, the numbers $-b_{j}$, $1\leq j \leq n'$ are in different gaps between $a_{1}>a_{2}>\cdots>a_{n}$. 
\item $m=1$, $n$ and $n'$ have the same parity.
\end{enumerate}

\end{thm}

\bigskip

\begin{rem}
\begin{enumerate}
\item In the case $(2)$ and $(3)$ of the above theorem, we can get similar results for the other parity situation. However, more notation and the CM periods will be involved so we neglect them here. Details can be found in Chapter $10$ and $11$ of \cite{linthesis}.
\item For the proof, we refer to \cite{harris97} for the case $(1)$, \cite{harrismotivic} for the case $n'=n-1$ in case $(2)$ and \cite{linthesis} for general cases.
\item These results have been generalized to CM fields (c.f.\cite{guerberoffhecke} for case (1) and \cite{linthesis} for general cases). More precisely, let $F$ be a CM field and $\Pi$ be a certain cuspidal representation of $GL_{n}(\AF)$. One can still define automorphic periods for $\Pi$. We have similar results on critical values for tensor product of two such representations. The main difficulty of this generalisation is to show that the automorphic periods factorise through infinite places. This factorization generalises an important conjecture of Shimura, and is predicted by motivic calculation. The proof can be found in the thesis and a forthcoming paper of the author.
\item One application of the above theorem is the functoriality of automorphic periods. For example, let $F/L$ be a cyclic extension of CM fields. We consider the base change of $GL_{n}$ with respect to $L/F$. If $\Pi$, a certain cuspidal representation of $GL_{n}(\AL)$, descends to $\pi$, a certain cuspidal representation of $GL_{n}(\AF)$ by base change, then there are relations between automorphic periods for $\Pi$ and those for $\pi$. We refer to Chapter $8$ of \cite{linthesis} for the details.
\end{enumerate}

\end{rem}

\bibliography{bibfile}
\bibliographystyle{alpha}

\end{document}